\documentclass[12pt]{amsart}

\usepackage{amsthm}
\usepackage{mathtools}
\usepackage{amssymb}
\usepackage{amsmath}
\usepackage{graphicx}

\usepackage{color}
\usepackage[colorlinks=true,linkcolor=blue,citecolor=blue]{hyperref}
\usepackage{enumitem}

\usepackage{todonotes}

\usepackage{mathrsfs}

\newtheorem{theorem}{Theorem}[section]
\newtheorem*{theorem*}{Theorem}
\newtheorem{proposition}[theorem]{Proposition}
\newtheorem{corollary}[theorem]{Corollary}
\newtheorem{lemma}[theorem]{Lemma}
\theoremstyle{definition}

\theoremstyle{remark}

\newtheorem{remark}[theorem]{Remark}

\numberwithin{equation}{section}

\allowdisplaybreaks

\begin{document}

\title{Algebras of frequently hypercyclic vectors}

\thanks{The authors were supported by the Fonds de la Recherche Scientifique - FNRS, grant no. PDR T.0164.16. The first author was also supported by MINECO and FEDER Project MTM2017-83262-C2-1-P}



\author[J. Falc\'{o}]{Javier Falc\'{o}}
\address[Javier Falc\'{o}]{Departamento de An\'alisis Matem\'atico, Universidad de Valencia, Doctor Moliner 50, 46100 Burjasot (Valencia),
Spain.}
\email{Francisco.J.Falco@uv.es}
\author[K.-G. Grosse-Erdmann]{Karl-G. Grosse-Erdmann}
\address[Karl-G. Grosse-Erdmann]{D\'epartement de Math\'ematique, Institut Complexys,
Universit\'e de Mons, 20 Place du Parc, 7000 Mons, Belgium} \email{kg.grosse-erdmann@umons.ac.be}
%


\begin{abstract}
We show that the multiples of the backward shift operator on the spaces $\ell_{p}$, $1\leq p<\infty$, or $c_{0}$, when endowed with coordinatewise multiplication, do not possess frequently hypercyclic algebras. More generally, we characterize the existence of algebras of $\mathcal{A}$-hypercyclic vectors for these operators. We also show that the differentiation operator on the space of entire functions, when endowed with the Hadamard product, does not possess frequently hypercyclic algebras. On the other hand, we show that for any frequently hypercyclic operator $T$ on any Banach space, $FHC(T)$ is algebrable for a suitable product, and in some cases it is even strongly algebrable. 
\end{abstract}

\maketitle

One of the areas of research in linear dynamics is the study of the structure of the set of hypercyclic vectors 
$$
HC(T)=\{x\in X: \{x,Tx,T^2x,\ldots\} \text{ is dense in } X\}
$$
of a linear dynamical system $(X,T)$.

It is well known that, as soon as an operator is hypercyclic, its set of hypercyclic vectors has a rich structure. It always contains a dense subspace, and it sometimes even contains a closed infinite-dimensional subspace.  

When  we also have a multiplicative structure on the space $X$, it is natural to ask whether we can carry this structure to some subset of $HC(T)$. This problem was first studied in \cite{Aron07} for two classical operators on the space $H(\mathbb{C})$ of entire functions: the Birkhoff translation operators $T_a: f\to f(\cdot +a)$, $a\neq 0$, and the MacLane differentiation operator $D: f\to f'$. Aron et al.\ showed that no non-trivial power of any entire function is hypercyclic for a translation operator; however, for the differentiation operator, many hypercyclic vectors satisfy that all of its powers are hypercyclic. In fact, these elements form a dense $G_\delta$-subset of $H(\mathbb{C})$.

The  natural scenario for considering the existence of multiplicative structures are Banach or Fr\'echet algebras. In this paper we will focus mainly on Banach algebras. Recall that a Banach algebra is an (associative) algebra $X$ over the real or complex numbers that at the same time is also a Banach space where the norm and the multiplicative structure are related by the following inequality:
\begin{equation*}
\|x\,y\| \leq \|x\|\,\|y\|,\quad x,y\in X.
\end{equation*}
All the algebras that we are considering here are algebras over the field of the complex numbers. 

When we have a hypercyclic operator $T$ on a Banach algebra $X$ it is natural to ask if $HC(T)$ contains a non-trivial subalgebra of $X$, except zero. When such a subalgebra exists it is called a \textit{hypercyclic algebra} for $T$.

Motivated by the result  for the MacLane differentiation operator obtained in \cite{Aron07}, Shkarin \cite{Sh10} and Bayart and Matheron \cite[Theorem 8.26]{EtMAt} showed independently that it admits a hypercyclic algebra, thereby providing the first known example of an operator with a hypercyclic algebra. Recently, B\`es, Conejero and Papathanasiou \cite{BesConPap,BesConPap2,BesPap} and Bayart \cite{Bay} extended the existence of hypercyclic algebras to various other operators like convolution operators and certain translation operators, and the authors \cite{falgro} proved  the existence of algebras of hypercyclic vectors for weighted backward shifts on Fr\'echet sequence spaces that are algebras when endowed with coordinatewise multiplication or with the Cauchy product. 

Here we are interested in the study of the more restrictive case of frequent hypercyclicity; in continuation of our approach in \cite{falgro}, our methods in this paper will be constructive. A (continuous, linear) operator $T$ is called \textit{frequently hypercyclic} if there is some $x \in X$ so that, for any non-empty open subset $U$ of $X$,
$$
 \underline{\text{dens}}\{n\in\mathbb N_{0} : T^{n}x\in U\}>0,
$$ 
where the  lower density of a subset $A \subset \mathbb N_{0}$ is defined as 
$$
\underline{\text{dens}}(A)=\liminf_{N\to \infty} \frac{\text{card}\{0\leq n\leq N : n\in A\}}{N+1}.
$$
In this case, $x$ is called a \textit{frequently hypercyclic vector} for $T$. The set of
frequently hypercyclic vectors for $T$ is denoted by $FHC(T)$. \textit{Upper frequent hypercyclicity} is defined in an analogous way.
We refer to the monographs \cite{EtMAt} and \cite{GrPe11} for these and other notions of linear dynamics.

We say that a frequently hypercyclic operator $T$ on a Banach algebra $X$ admits a \textit{frequently hypercyclic algebra} if  $FHC(T)$ contains a non-trivial subalgebra of $X$, except zero. If a frequently hypercyclic algebra is not finitely generated then $FHC(T)$ is said to be algebrable. In this paper, when we say that an algebra is infinitely generated, we mean that it cannot be finitely generated. Finally, $FHC(T)$ is said to be \textit{strongly algebrable} if it contains an infinitely generated algebra, except zero, that is isomorphic to a free algebra. For a commutative Banach algebra, this is equivalent to saying that $FHC(T)$ contains an algebra, except zero, that is generated by an algebraically independent sequence $(x_n)_n$. The notion of strong algebrability was introduced by Bartoszewicz and G{\l}a\c{b} in \cite{ArtGlab}. For more details on these topics we refer to the monograph \cite{ABPS16}.

The starting point of this paper is the study of the Rolewicz operator $\lambda B$, for $\lambda\in\mathbb C$ with $\vert \lambda\vert >1$, and its set of frequently hypercyclic vectors. Recall that for the Banach space $X=\ell_{p}$, $1\leq p<\infty$, or $X=c_{0}$ and any $\lambda \in \mathbb C$ the operator $\lambda B$ is defined by 
$$
\lambda B(x(1), x(2), x(3), \ldots) = (\lambda x(2), \lambda x(3), \lambda x(4), \ldots),\quad x\in X,
$$
and it is frequently hypercyclic if $\vert \lambda \vert >1$. The fact that these operators have dense orbits was first proved by Rolewicz \cite{Rol}.

In Section \ref{s-rol} we show that, under the coordinatewise multiplicative structure, the Banach algebras $\ell_{p}$ and $c_0$ do not have a frequently hypercyclic algebra for the Rolewicz operator. Motivated by this we ask in the remainder of the paper if positive results exist in some weaker or related contexts.

We begin by demanding less than frequent hypercyclicity. In Section \ref{s-ahyp} we will recall the notion of an $\mathcal A$-hypercyclic vector. We then characterize the Furstenberg families $\mathcal{A}$ for which the Rolewicz operator on the spaces $\ell_{p}$ or $c_{0}$ admits an algebra of $\mathcal A$-hypercyclic vectors, except 0. In particular, this is true when one replaces lower density (as used in frequent hypercyclicity) by lower logarithmic density. This implies that the Rolewicz operator has an upper frequently hypercyclic algebra, which was also shown, in a different context and with different methods, by Bayart, Costa J\'unior and Papathanasiou \cite{BaCoPa}.

In Section \ref{s-macl} we wonder if there exist frequently hypercyclic algebras for other weighted shifts. We only get further negative results, first for some other natural weights in the case of $\ell_{p}$, and then for the MacLane operator on $H(\mathbb{C})$ when endowed with the Hadamard product.

In Section \ref{s-alloperators} it is shown that every Banach space admits a product under which it becomes a Banach algebra and for which every frequently hypercyclic operator $T$ satisfies that $FHC(T)$ is algebrable. A stronger result holds under the assumption that $T$ has a frequently hypercyclic subspace; in this case, $FHC(T)$ is even strongly algebrable under a suitable product.

\section{The Rolewicz operator and the coordinatewise multiplicative structure}\label{s-rol}

The first natural multiplicative structure that we can consider on the Banach spaces $\ell_{p}$, $1\leq p < \infty$, and $c_0$ is coordinatewise multiplication. It is well known that under this product the spaces turn into Banach algebras. We show in this section that the set $FHC(\lambda B)$ does not contain any algebra of frequently hypercyclic vectors, except zero. 

The following easy lemma will be the key point to obtain this result. Indeed we will only need to use this lemma in the case of $j=1$. 

\begin{lemma}\label{lemmapld}
Let  $A=(n_{k})_{k\geq 1}$ be a strictly increasing sequence of natural numbers. If $A$ has positive lower density, then there exist numbers $M_j\in\mathbb N$, $j\in\mathbb N$, such that 
\begin{equation*}
\label{boundedcondition}
n_{k+j}\leq M_{j}n_{k}
\end{equation*}
 for all $j,k\in\mathbb N$. 
\end{lemma}

\begin{proof}
Since $\underline{\text{dens}}(A)=\liminf_{k\to\infty}\frac{k}{n_{k}}>0$, the sequence $(\frac{n_{k}}{k})_{k}$ is bounded. Hence there exists a natural number $M$ with $n_{k}\leq Mk$ for all $k\in\mathbb N$. Then, for all $k\in\mathbb N$,
$$
n_{k+1}\leq M(k+1)\leq 2Mk\leq 2M n_{k},
$$
which implies the result when we take $M_j=(2M)^j$, $j\in\mathbb N$.
\end{proof}

As usual, $e_n$ denotes the sequence $(0,\ldots,0,1,0,\ldots)$ with the 1 at index $n$.

\begin{proposition}\label{prop:nopowers}
Let us consider the Banach algebra $X=\ell_{p}$, $1\leq p<\infty$, or $X=c_{0}$ endowed with coordinatewise multiplication, and let $\lambda\in\mathbb C$ with $\vert \lambda\vert>1$. If $x\in FHC(\lambda B)$, then there exists a natural number $M$ such that $x^{m}\notin HC(\lambda B)$ for any $m\geq M$.
\end{proposition}

\begin{proof}
Fix $0< \varepsilon<1$. Since $x\in FHC(\lambda B)$ the set 
$$
A=\{n\in\mathbb N:\Vert (\lambda B)^{n}x\Vert<\varepsilon\}
$$
has positive lower density. We enumerate the elements of $A$ as a strictly increasing sequence $(n_{k})_{k}$. By Lemma \ref{lemmapld} there exists a natural number $M$ such that 
\begin{equation}\label{equation:ineqcond}
n_{k+1}\leq Mn_{k}
\end{equation}
for all $k\in\mathbb{N}$. We claim that $x^{m}\notin HC(\lambda B)$ for any $m\geq M$. We proceed by contradiction. Assume this is not the case. Fix $m\geq M$ with $x^{m}\in HC(\lambda B)$. Then there exists a natural number $n\geq n_1$ with 
$$
\Vert (\lambda B)^{n}x^{m}-e_1\Vert <1-\varepsilon^m.
$$
Then we have that
$$
1-\varepsilon^m>\Vert (\lambda B)^{n}x^{m}-e_{1}\Vert \geq \vert \lambda^{n}x^{m}(n+1)-1\vert \geq 1-\vert \lambda^{n}x^{m}(n+1)\vert,
$$
hence $\vert \lambda^{n}x^{m}(n+1)\vert>\varepsilon^m$ and as a consequence
$$
\vert x(n+1)\vert>\frac{\varepsilon}{\vert \lambda\vert^{\frac{n}{m}}}.
$$

Now, there is some $k\in\mathbb{N}$ such that $n_k\leq n <n_{k+1}$. Equation \eqref{equation:ineqcond} implies that
$n < M n_k\leq mn_k$, so that
\[
mn_k-n\geq 0.
\]

But then
\begin{align*}
\varepsilon 
&>\Vert (\lambda B)^{n_{k}}x\Vert \geq \big\vert \big[(\lambda B)^{n_{k}}x\big](n+1-n_{k})\big\vert\\
&=\vert \lambda^{n_{k}}x(n+1)\vert\\
&>\vert \lambda\vert^{n_{k}}\frac{\varepsilon}{\vert \lambda \vert^{\frac{n}{m}}}\\
&=\varepsilon\vert\lambda\vert^{\frac{mn_{k}-n}{m}}\geq \varepsilon,
\end{align*}
which is a contradiction. The proof is complete.
\end{proof}

It is proved in \cite{falgro} that in the present setting the set of hypercyclic vectors for the Rolewicz operator is algebrable. As a consequence of the proposition we see that no hypercyclic algebra for the Rolewicz operator can contain a frequently hypercyclic vector.

\begin{corollary} 
\label{coro:noalgebra}
Let us consider the Banach algebra $X=\ell_{p}$, $1\leq p<\infty$, or $X=c_{0}$ endowed with coordinatewise multiplication. Then for any $\lambda\in\mathbb C$ with $\vert \lambda\vert>1$ the Rolewicz operator $\lambda B$ does not have any frequently hypercyclic algebra. Furthermore, no hypercyclic algebra for $\lambda B$ can contain a frequently hypercyclic vector.
\end{corollary}

\begin{remark} 
One might wonder whether Proposition \ref{prop:nopowers} can be further strengthened: could it be that if $x$ is frequently hypercyclic for $\lambda B$ then $x^m$ is not even supercyclic for sufficiently large $m$? Recall that a vector $x\in X$ is called \textit{supercyclic} for an operator $T$ on $X$ if the set $\{\alpha T^nx : \alpha \in\mathbb{C}, n\in \mathbb{N}_0\}$ is dense in $X$. The supposed strengthening of the proposition fails dramatically: in our setting, any power $x^m$, $m\in \mathbb{N}$, of any supercyclic vector $x$ (and hence also of any frequently hypercyclic vector) is supercyclic. Indeed, let $x$ be supercyclic for $\lambda B$. Let $m\in \mathbb{N}$, and let $y\in X$ be a finite sequence. Choose a vector $z\in X$ such that $z^m=y$. Then there exist sequences $(\alpha_k)_k$ of complex numbers and $(n_k)_k$ of natural numbers such that $\alpha_k (\lambda B)^{n_k} x \to z$ as $k\to\infty$. It follows from the continuity of taking powers in $X$ that $\alpha_k^m (\lambda^{(m-1)n_k}) (\lambda B)^{n_k} x^m \to y$ as $k\to\infty$. This implies that $x^m$ is supercyclic for $\lambda B$.
\end{remark}

\section{The Rolewicz operator and $\mathcal{A}$-hypercyclicity}\label{s-ahyp}

Even though the algebras $X=\ell_{p}$, $1\leq p<\infty$, or $X=c_{0}$ under coordinatewise multiplication cannot contain an algebra of frequently hypercyclic vectors for the Rolewicz operator $\lambda B$, a weaker result holds. We show that there are algebras of upper frequently hypercyclic vectors. We obtain this result in the broader context of $\mathcal{A}$-hypercyclicity. Recall that a non-empty family $\mathcal{A}$ of subsets of $\mathbb{N}_0$ is called a \textit{Furstenberg family} if $\varnothing\notin \mathcal{A}$ and if $A\in\mathcal{A}$ and $A\subset B\subset\mathbb{N}_0$ implies that $B\in\mathcal{A}$. The Furstenberg family $\mathcal{A}$ is called \textit{finitely invariant (f.i.)} if, for any $A\in\mathcal{A}$, $A\setminus [0,N]\in\mathcal{A}$ for any $N\geq 0$. Now, an operator $T$ on a Banach space $X$ is called \textit{$\mathcal{A}$-hypercyclic} if there is a vector $x\in X$, called \textit{$\mathcal{A}$-hypercyclic for $T$}, such that, for any non-empty open set $U\subset X$,
\[
\{n\geq 0 : T^n x\in U\}\in \mathcal{A}.
\]
We refer to \cite{BMPP16} and \cite{BoGE17} for these notions. 

Frequent hypercyclicity is $\mathcal{A}$-hypercyclicity for the family $\mathcal{A}$ of sets of positive lower density, while upper frequent hypercyclicity is defined by the family of sets of positive upper density. 

We start with the following very general observation.

\begin{proposition}\label{prop:rulethemall}
Let $X=\ell_{p}$, $1\leq p<\infty$, or $X=c_{0}$, be considered as a Banach algebra under coordinatewise multiplication. Let $\mathcal{A}$ be an f.i.\ Furstenberg family, $x_{0}\in X$ and $m\in \mathbb{N}$. If $x_{0}^{m}$ is $\mathcal{A}$-hypercyclic for $\lambda B$, $|\lambda|>1$, then so
is $\sum_{\nu=m}^{N}\alpha_{\nu}x_{0}^{\nu}$ for any $\alpha_{m},\ldots,\alpha_{N}\in\mathbb C$, $N\geq m$, with $\alpha_m\neq 0$. 
\end{proposition}

\begin{proof}
Since for $\alpha\neq 0$, $\alpha x$ is $\mathcal{A}$-hypercyclic if and only if $x$ is, we may assume that $\alpha_m=1$. Then the fact that the backward shift is multiplicative implies that
\begin{align*}
(\lambda B)^n\Big(x_0^m+\sum_{\nu=m+1}^{N}\alpha_{\nu}x_0^\nu\Big)- (\lambda B)^n x_0^m &=\sum_{\nu=m+1}^{N}(\lambda B)^n(\alpha_{\nu}x_{0}^{\nu})\\
&= \sum_{\nu=m+1}^{N}\alpha_{\nu}B^n (x_{0}^{\nu-m})(\lambda B)^n(x_{0}^{m})\\
&=(\lambda B)^n (x_0^m) \sum_{\nu=m+1}^{N}\alpha_{\nu}B^n (x_{0}^{\nu-m}).
\end{align*}
Now, since $B^n x\to 0$ as $n\to\infty$ for any $x\in X$, we have that
\[
\sum_{\nu=m+1}^{N}\alpha_{\nu}B^n (x_{0}^{\nu-m}) \to 0
\]
in $X$ as $n\to\infty$. Then the submultiplicativity of the norm in $X$ easily implies that $x_0^m+\sum_{\nu=m+1}^{N}\alpha_{\nu}x_0^\nu$ is $\mathcal{A}$-hypercyclic as soon as $x_0^m$ is.
\end{proof}

\begin{remark} 
We note that the result holds in even greater generality. It remains true, with essentially the same proof, for any weighted backward shift $B_w$ on any Fr\'echet sequence algebra $X$ under coordinatewise multiplication provided only that the (unweighted) backward shift $B$ is an operator on $X$ so that $B^nx\to 0$ for all $x\in X$; see \cite{falgro} for the relevant notions.
\end{remark}

The proposition applies, in particular, to frequent hypercyclicity.

\begin{corollary}
Let $X=\ell_{p}$, $1\leq p<\infty$, or $X=c_{0}$, be considered as a Banach algebra under coordinatewise multiplication. Let $x_0\in X$ and $m\in \mathbb{N}$. If $x_{0}^{m}$ is frequently hypercyclic for $\lambda B$, $|\lambda|>1$, then so is $\sum_{\nu=m}^{N}\alpha_{\nu}x_{0}^{\nu}$ for any $\alpha_{m},\ldots,\alpha_{N}\in\mathbb C$, $N\geq m$, with $\alpha_m\neq 0$.
\end{corollary}

We can now obtain a characterization of the f.i.\ Furstenberg families $\mathcal{A}$ for which $\lambda B$ admit an algebra of $\mathcal{A}$-hypercyclic vectors.

\begin{theorem}\label{thm:charahcalg}
Let $X=\ell_{p}$, $1\leq p<\infty$, or $X=c_{0}$, be considered as a Banach algebra under coordinatewise multiplication. Let $\mathcal{A}$ be an f.i.\ Furstenberg family and $|\lambda|>1$. Then $\lambda B$ admits an algebra of $\mathcal{A}$-hypercyclic vectors, except zero, if and only if there exists a family $(A(l,m))_{l,m\geq 1}$ of pairwise disjoint sets in $\mathcal{A}$ such that, for any $l,m,l',m'\geq 1$ and for any $n\in A(l,m)$, $n'\in A(l',m')$ with $n'>n$, we have that
\begin{equation}\label{charcond}
n'\geq n+l\quad\text{and}\quad n'\frac{m}{m'}\geq n+l+l'.
\end{equation}
\end{theorem}

\begin{proof} We first show that the condition is necessary. To see this, we write
\begin{equation}\label{rho}
\rho_{l,m} = |\lambda|^{lm}+1, \quad l,m\geq 1,
\end{equation}
and we choose $\varepsilon>0$ such that $3\varepsilon < \min(1,|\lambda|^{-1})$. We consider the points
\begin{equation}\label{zed}
z_{l,m} = \rho_{l,m}e_1 + 2 \varepsilon^{lm} \sum_{k=2}^{l+1} e_k \in X, \quad l,m\geq 1.
\end{equation}

By assumption there is a point $x\in X$ so that all of its powers $x^m$, $m\geq 1$, are $\mathcal{A}$-hypercyclic for $\lambda B$. Hence, for any $l,m\geq 1$, there is a set $A(l,m)\in\mathcal{A}$ such that, for any $n\in A(l,m)$,
\begin{equation}\label{approx}
\|(\lambda B)^n x^m - z_{l,m}\|\leq \varepsilon^{lm}.
\end{equation}
Since $\mathcal{A}$ is finitely invariant we may assume that 
\begin{equation}\label{condfi}
|\lambda|^{\frac{n}{m^2}} > 3\varepsilon^{-l}, \quad n\in A(l,m).
\end{equation}

First we prove the second condition in \eqref{charcond}. Fix  $l,m,l',m'\geq 1$, $n\in A(l,m)$ and $n'\in A(l',m')$ with $n'>n$. Then we have by \eqref{zed} and \eqref{approx} that
\[
|\lambda^{n'} x^{m'}(n'+1) - \rho_{l',m'}|\leq \varepsilon^{l'm'}\leq 1,
\]
so that by \eqref{rho} and the triangle inequality we obtain that
\[
(|\lambda|^{l'm'}+1) - |\lambda^{n'} x^{m'}(n'+1)| \leq 1.
\]
Hence
\[
|\lambda^{n'} x^{m'}(n'+1)|\geq |\lambda|^{l'm'},
\]
and therefore
\begin{equation}\label{ix}
|x(n'+1)|\geq |\lambda|^{l'-\frac{n'}{m'}}.
\end{equation}

On the other hand, \eqref{zed} and \eqref{approx} also imply that, if $n<n'\leq n+l$,
\[
|\lambda^{n} x^{m}(n'+1) - 2\varepsilon^{lm}|\leq \varepsilon^{lm}
\]
and if $n'> n+l$ then
\[
|\lambda^{n} x^{m}(n'+1)|\leq \varepsilon^{lm};
\]
since $3\varepsilon \leq |\lambda|^{-1}$, we deduce that for any $n'>n$
\[
|\lambda^{n} x^{m}(n'+1)|\leq 3\varepsilon^{lm} \leq (3\varepsilon)^{lm} \leq |\lambda|^{-lm}
\]
and therefore
\begin{equation}\label{ix2}
|x(n'+1)|\leq |\lambda|^{-l-\frac{n}{m}}.
\end{equation}

From \eqref{ix} and \eqref{ix2} we obtain that
\[
|\lambda|^{\frac{n'}{m'}-l'}\geq |\lambda|^{\frac{n}{m}+l},
\]
hence
\[
\frac{n'}{m'}\geq \frac{n}{m}+l +l',
\]
which implies the second condition in \eqref{charcond}. This condition also implies the first condition in \eqref{charcond} in the case that $m\leq m'$. 

For the remainder of the proof of necessity we will first show that if $l,m,l',m'\geq 1$, $n\in A(l,m)$ and $n'\in A(l',m')$, then 
\begin{equation}\label{hypo}
n\leq n'\leq n+l-1.
\end{equation}
implies that
\begin{equation}\label{ix5}
|\lambda|^{n(\frac{1}{m'}-\frac{1}{m})}\leq 3 \varepsilon^{l'-l}.
\end{equation}
Indeed, since then $n+2\leq n'+2\leq n+l+1$, we get from \eqref{zed} and \eqref{approx} that
\[
|\lambda^{n} x^{m}(n'+2) - 2\varepsilon^{lm}|\leq \varepsilon^{lm},
\]
hence 
\[
|\lambda^{n} x^{m}(n'+2)|\geq \varepsilon^{lm}
\]
and therefore 
\begin{equation}\label{ix3}
|x(n'+2)|\geq \varepsilon^{l}|\lambda|^{-\frac{n}{m}}.
\end{equation}

On the other hand, we get from \eqref{zed} and \eqref{approx} that
\[
|\lambda^{n'} x^{m'}(n'+2) - 2\varepsilon^{l'm'}|\leq \varepsilon^{l'm'},
\]
hence 
\[
|\lambda^{n'} x^{m'}(n'+2)|\leq 3\varepsilon^{l'm'}
\]
and therefore 
\begin{equation}\label{ix4}
|x(n'+2)|\leq 3^{\frac{1}{m'}}\varepsilon^{l'}|\lambda|^{-\frac{n'}{m'}}\leq 3 \varepsilon^{l'}|\lambda|^{-\frac{n}{m'}}.
\end{equation}
Combining \eqref{ix3} and \eqref{ix4} we obtain \eqref{ix5}.

In particular, \eqref{hypo} implies that $m'\geq m$. Indeed, if $m'< m$, then
\[
\frac{1}{m'}-\frac{1}{m}\geq \frac{1}{m^2},
\]
and hence, by \eqref{ix5}, 
\[
|\lambda|^{\frac{n}{m^2}}\leq 3 \varepsilon^{-l},
\]
which contradicts \eqref{condfi}.

We can now show that the first condition in \eqref{charcond} also holds in the case that $m'< m$. By the above, \eqref{hypo} cannot hold, and thus $n'> n$ implies that $n'\geq n+l$.

This concludes the proof of condition \eqref{charcond}.

It remains to show that the sets $A(l,m)$, $l,m\geq 1$, are pairwise disjoint. Indeed, suppose that $n\in A(l,m)$, $n'\in A(l',m')$ with $n=n'$ and $(l,m)\neq (l',m')$. If $m= m'$, then we may assume without loss of generality that $l'>l$. But since \eqref{hypo} holds, we have inequality \eqref{ix5}, which now implies that $1\leq 3\varepsilon$, which contradicts the choice of $\varepsilon$. Thus $m\neq m'$, and we may assume without loss of generality that $m'<m$. But then we have already seen that \eqref{hypo} leads to a contradiction. 

This finishes the proof of necessity. As for sufficiency, it suffices, by Proposition \ref{prop:rulethemall}, to show that there is some point $x\in X$ so that all of its powers $x^m$, $m\geq 1$, are $\mathcal{A}$-hypercyclic. To this end, let $(y_l)_{l\geq 1}$ be a dense sequence of finite sequences in $X$ such that $\|y_l\|\leq l$ and $s_l\leq l$, where $s_l$ is the largest index of the non-zero coordinates of $y_l$. For $m\geq 1$, let $y_l^{\frac{1}{m}}$ be any fixed $m$-th root of $y_l$, and we denote by $y_l^{\frac{j}{m}}$, $j\geq 1$, the $j$th power of the latter number. We adopt the same convention for $\lambda^{\frac{1}{m}}$ and $\lambda^{\frac{j}{m}}$.

We first note that 
\begin{equation}\label{yl}
\|y_l^{\frac{j}{m}}\| \leq l^{\max(\frac{j}{m},1)}
\end{equation}
for all $l,j,m\geq 1$. For $X=c_0$ this is obvious. For $X=\ell_p$, $1\leq p<\infty$, if $j\geq m$ this follows immediately from the fact that $\|y_l\|_\beta\leq \|y_l\|_\alpha$ whenever $\beta\geq \alpha\geq 1$; if $j\leq m$ then this follows from H\"older's inequality:
\[
\sum_{k=1}^{s_l} |y_l(k)|^{p\frac{j}{m}} \leq \Big(\sum_{k=1}^{s_l} |y_l(k)|^p\Big)^{\frac{j}{m}} {s_l}^{1-\frac{j}{m}} \leq l^{p\frac{j}{m}} l^{1-\frac{j}{m}}\leq l^p.
\]

Now let $(A(l,m))_{l,m\geq 1}$ be a family of pairwise disjoint sets in $\mathcal{A}$ that satisfies the hypothesis. In replacing $(A(l,m))_{l,m\geq 1}$ by $(A(l+m,m))_{l,m\geq 1}$ we see that we may assume that, for any $l,m,l',m'\geq 1$, and for any $n\in A(l,m)$, $n'\in A(l',m')$ with $n'>n$, we have that
\begin{equation}\label{charcond2}
n'\geq n+l\quad\text{and}\quad n'\frac{m}{m'}\geq n+m'+l+l'.
\end{equation}
In addition, since $\mathcal{A}$ is finitely invariant, we may assume that, for any $l,m\geq 1$, 
\begin{equation}\label{fi2}
\sum_{n\in A(l,m)} \frac{1}{|\lambda|^{\frac{n}{m}}}\leq \frac{1}{l2^{l+m}}.
\end{equation}

We may now construct the desired sequence $x$. Denoting by $F$ the forward shift
\[
F(x(1),x(2),x(3),\ldots) = (0, x(1), x(2),\ldots)
\]
we define
\[
x=\sum_{l,m\geq 1} \sum_{n\in A(l,m)} \frac{1}{\lambda^\frac{n}{m}} F^n y_l^{\frac{1}{m}}.
\]
Since, with \eqref{yl},
\[
 \sum_{n\in A(l,m)} \Big\|\frac{1}{\lambda^\frac{n}{m}} F^n y_l^{\frac{1}{m}}\Big\| \leq l\sum_{n\in A(l,m)} \frac{1}{|\lambda|^\frac{n}{m}},
\]
it follows from \eqref{fi2} that the series defining $x$ converges in $X$. 

It remains to show that, for every $m'\geq 1$, the power $x^{m'}$ of $x$ is $\mathcal{A}$-hypercyclic. To see this, let $l'\geq 1$ and $n'\in A(l',m')$. Then
\[
T^{n'}x^{m'} = y_{l'} + \sum_{l,m\geq 1}\sum_{\substack{n\in A(l,m)\\n>n'}} \frac{\lambda^{n'}}{\lambda^{n\frac{m'}{m}}} F^{n-n'} y_l^{\frac{m'}{m}};
\]
here we have used that the sets $A(l,m)$ are pairwise disjoint, that the terms with $n<n'$ vanish since $n'-n\geq l\geq s_l$ for any $n\in A(l,m)$, $l,m\geq 1$, and that the operator $F$ is multiplicative so that it commutes with taking powers. 

We therefore conclude that
\begin{align*}
\|T^{n'}x^{m'}-y_{l'}\| &\leq \sum_{l,m\geq 1}\sum_{\substack{n\in A(l,m)\\n>n'}} \frac{1}{|\lambda|^{n\frac{m'}{m}-n'}} \|y_l^{\frac{m'}{m}}\|\\
&\leq \sum_{l,m\geq 1}\Big(\sum_{k\geq m+l+l'} \frac{1}{|\lambda|^{k}}\Big) l^{\max(\frac{m'}{m},1)}\quad (\text{by \eqref{charcond2} and \eqref{yl}})\\
&=C\sum_{l,m\geq 1}\frac{1}{|\lambda|^{m+l+l'}} l^{\max(\frac{m'}{m},1)}\\
&\leq C\Big( \sum_{1\leq m < m'} \sum_{l\geq 1}\frac{l^{m'}}{|\lambda|^{l}} + \sum_{m\geq m'}\frac{1}{|\lambda|^{m}}\sum_{l\geq 1}\frac{l}{|\lambda|^{l}}\Big) \frac{1}{|\lambda|^{l'}}\\
&=C_{m'}\frac{1}{|\lambda|^{l'}},
\end{align*}
where $C=\frac{1}{1-|\lambda|^{-1}}$ and $C_{m'}$ is a constant that only depends on $\lambda$ and $m'$. Now, since $\frac{1}{|\lambda|^{l'}}\to 0$ as $l'\to\infty$, the sequence $(y_l)$ is dense in $X$, $A(m',l')\in\mathcal{A}$ for all $m',l'\geq 1$ and $\mathcal{A}$ is f.i., we deduce that $x^{m'}$ is $\mathcal{A}$-hypercyclic, as had to be shown.
\end{proof} 

It is easy to see that the theorem applies to upper frequent hypercyclicity. But the conclusion even holds for a stronger notion of $\mathcal{A}$-hypercyclicity that is defined by some lower type density. Indeed, the lower logarithmic density of a set $A\subset\mathbb{N}_0$ is defined by
\[
\text{log-}\underline{\text{dens}} (A) = \liminf_{N\to\infty}\frac{\sum_{1\leq n\leq N, n\in A}\frac{1}{n}}{\sum_{1\leq n\leq N}\frac{1}{n}}.
\]
Let $\mathcal{A}_{\underline{\log}}$ be the Furstenberg family of all sets of positive lower logarithmic density, which is obviously finitely invariant. Now, since for any subset $A\subset\mathbb{N}_0$,
\[
\text{log-}\underline{\text{dens}} (A)\leq \overline{\text{dens}} (A),
\]
see \cite[Lemma 2.8]{ErMo17}, we have that $\mathcal{A}_{\underline{\log}}$-hypercyclicity implies upper frequent hypercyclicity.

\begin{corollary} \label{cor:loghcalg}
Let $X=\ell_{p}$, $1\leq p<\infty$, or $X=c_{0}$, be considered as a Banach algebra under coordinatewise multiplication. Then, for any $|\lambda|>1$, $\lambda B$ admits an algebra of $\mathcal{A}_{\underline{\log}}$-hypercyclic vectors, except zero. In particular, it admits an algebra of upper frequently hypercyclic vectors, except zero.
\end{corollary}

\begin{proof} We may work equivalently with the density
\[
\underline{d}(A) = \liminf_{N\to\infty}\frac{\sum_{2\leq n\leq N, n\in A}(\log_2(n)-\log_2(n-1))}{\log_2(N)}.
\]
This follows from the fact 
\[
\frac{\log_2(n)-\log_2(n-1)}{1/n}
\]
converges to a strictly positive number as $n\to\infty$, so that the sets of positive density for $\underline{d}$ and $\text{log-}\underline{\text{dens}}$ coincide. 

For $l,m\geq 1$, let $I(l,m)$ be the set of all numbers $n\in\mathbb{N}$ whose dyadic representation $n=\sum_{j=0}^\infty a_j2^j=:(a_0,a_1,a_2,\ldots)$ has the form
\[
n = (0,\ldots, 0, 1,\ldots, 1, 0, \ast)
\]
with $l-1$ leading zeros, followed by $m$ ones, then one zero, followed by an arbitrary tail, see \cite[proof of Lemma 9.5]{GrPe11}. Since 
\[
\frac{\frac{2^{2^r}}{m}-l}{2^{2^{r-1}.3/2}}\to\infty
\]
as $r\to\infty$, there is some $r_{l,m}$ such that
\[
\frac{\frac{2^{2^r}}{m}-l}{2^{2^{r-1}.3/2}}\geq 1
\]
for all $r\geq r_{m,l}$.

We then define
\[
A(l,m) = \bigcup_{r\in I(l,m), r\geq r_{l,m}} B(l,r) 
\]
with
\[
B(l,r) = \{ 2^{2^r}+2l, 2^{2^r}+4l, \ldots, 2^{2^r}+2N_{l,r}l\},
\]
where $N_{l,r}$ is such that $2^{2^r}+2(N_{l,r}+1)l\leq 2^{2^r.3/2}<2^{2^r}+2(N_{l,r}+2)l$. Since the sets $I(l,m)$, $l,m\geq 1$, and the sets $[2^{2^r},2^{2^r.3/2}]$, $r\geq 1$, are pairwise disjoint, the same is true for the sets $A(l,m)$, $l,m\geq 1$.

We first show that, for any $l,m\geq 1$, $\underline{d}(A(l,m))>0$. For this it suffices to calculate the proportion of elements from $A(l,m)$ in the set of natural numbers up to $2^{2^r}+2kl-1$ for $k=1,\ldots, N_{l,r}$, where $r\in I(l,m)$ is large; note that $2^{2^r}+2kl$ is the $k$th element of $B(l,r)$. Since the difference between two consecutive elements of $I(l,m)$ is $\rho:=2^{m+l}$, the block to the left of $B(l,r)$ is $B(l,r-\rho)$. By regarding only the elements in this block, we find that the mentioned proportion is at least
\begin{align*}
\frac{\sum_{s=1}^{N_{l,r-\rho}}(\log_2(2^{2^{r-\rho}}+2sl)-\log_2(2^{2^{r-\rho}}+2sl-1))}{\log_2(2^{2^r}+2kl-1)}\\
\geq \frac{\sum_{s=1}^{N_{l,r-\rho}}(\log_2(2^{2^{r-\rho}}+2sl)-\log_2(2^{2^{r-\rho}}+2sl-1))}{\log_2(2^{2^r.3/2})}.
\end{align*}
Since $(\log_2(n)-\log_2(n-1))_n$ is decreasing, this quotient is at least
\begin{align*}
&\frac{\frac{1}{2l}\sum_{\nu=2l+1}^{2(N_{l,r-\rho}+1)l}(\log_2(2^{2^{r-\rho}}+\nu)-\log_2(2^{2^{r-\rho}}+\nu-1))}{\log_2(2^{2^r.3/2})}\\
  =& \frac{1}{2l}\frac{\log_2(2^{2^{r-\rho}}+2(N_{l,r-\rho}+1)l)-\log_2(2^{2^{r-\rho}}+2l)}{\log_2(2^{2^r.3/2})}\\
	\geq& \frac{1}{2l}\frac{\log_2(2^{2^{r-\rho}.3/2}-2l)-\log_2(2^{2^{r-\rho}}+2l)}{\log_2(2^{2^r.3/2})},
\end{align*}
where in the last inequality we have used the definition of $N_{l,r-\rho}$. The last term tends to $\frac{1}{6l2^{\rho}}$ as $r\to\infty$, which gives us that $\underline{d}(A(l,m))>0$.

Now let $n\in A(l,m)$ and $n'\in A(l',m')$ with $n'>n$, for some $l,m,l',m'\geq 1$. 

If $(l,m)=(l',m')$ then we have immediately that $n'-n\geq 2l$, which implies \eqref{charcond} in this case. 

If $(l,m)\neq (l',m')$ then we have by definition of $A(l,m)$ that also $n'\geq n+l$. Moreover, $n'\in B(l',r')$ for some $r'\in I(l',m')$ with $r'\geq r_{l',m'}$, hence $n'\geq 2^{2^{r'}}$. Also, $n\in B(l,r)$ for some $r\in I(l,m)$. Since $I(l,m)$ and $I(l',m')$ are disjoint and $n<n'$, we have that $r\leq r'-1$ and hence $n+l\leq 2^{2^r.3/2}\leq 2^{2^{r'-1}.3/2}$. This implies that 
\[
\frac{\frac{n'}{m'}-l'}{n+l} \geq  \frac{\frac{2^{2^{r'}}}{m'}-l'}{2^{2^{r'-1}.3/2}},
\]
which is at least 1 by definition of $r_{l',m'}$. Hence we have that
\[
n'\frac{m}{m'}\geq \frac{n'}{m'}\geq n+l+l',
\]
which confirms \eqref{charcond} also in that case.
\end{proof}

\begin{remark} The corollary can be further strengthened. For any $m\geq 1$, consider the densities
\[
\log^{m}\text{-}\underline{\text{dens}}(A) = \liminf_{N\to\infty}\frac{\sum_{1\leq n\leq N, n\in A}\frac{\log^{m-1}(n)}{n}}{\sum_{1\leq n\leq N}\frac{\log^{m-1}(n)}{n}},
\]
which include the lower logarithmic density when $m=1$. Then $\log^{m}\text{-}\underline{\text{dens}}(A)$ decreases as $m$ increases, see \cite[Lemma 2.8]{ErMo17}, so that the corresponding $\mathcal{A}_{\underline{\log}^m}$-hypercyclicities get stronger with increasing $m$. 

Now, the derivative of $x\to\log_2^{m}(x)$ is $x\to \frac{\log^{m-1}(x)}{x}$, up to a multiplicative constant. This implies that $\log^{m}\text{-}\underline{\text{dens}}$ has the same sets of positive density as
\[
\underline{d}_m(A)=\liminf_{N\to\infty}\frac{\sum_{2\leq n\leq N, n\in A}(\log_2^{m}(n)-\log_2^m(n-1))}{\log_2^m(N)},\ A\subset \mathbb{N}_0.
\]
One may now repeat the proof of the corollary to show that the same result also holds for the notion of $\mathcal{A}_{\underline{\log}^m}$-hypercyclicity, for any $m\geq 1$.

One might then be tempted to consider the densities
\[
\liminf_{N\to\infty}\frac{\sum_{1\leq n\leq N, n\in A}\frac{1}{n^\alpha}}{\sum_{1\leq n\leq N}\frac{1}{n^\alpha}}
\]
for $\alpha<1$ close to 1. But, interestingly, Ernst and Mouze \cite[Lemma 2.10]{ErMo17} have shown that the corresponding $\mathcal{A}$-hypercyclicities all coincide with frequent hypercyclicity, so that the corollary fails for these densities.
\end{remark}

\section{The MacLane operator and frequently hypercyclic algebras}\label{s-macl}

If the Rolewicz operator does not have frequently hypercyclic algebras, do other weighted shifts have them? In this section, we only offer some more negative results.

We first look at the MacLane operator $D:f\to f'$ of differentiation on the space $H(\mathbb{C})$ of entire functions, which is endowed with its natural topology of locally uniform convergence. It is well known that $D$ is frequently hypercyclic, see \cite{BaGr}. When we identify $H(\mathbb{C})$ with the space of sequences of Taylor coefficients at 0 of entire functions, $D$ becomes a weighted backward shift with weights $w_n=n$, $n\geq 1$:
\[
D(x(0),x(1),x(2),\ldots) = (x(1),2x(2),3x(3),\ldots),\quad x \in H(\mathbb{C}).
\]
We turn $H(\mathbb{C})$ into a Fr\'echet algebra by endowing it with the product of coordinatewise multiplication, also called the Hadamard product on $H(\mathbb{C})$; see \cite{falgro} for more details.

Now, the proof of Proposition \ref{prop:nopowers} carries over to this setting.

\begin{proposition}\label{prop:nopowersmacl}
Let $D$ denote the MacLane operator. If $x\in FHC(D)$, then there exists a natural number $M$ such that $x^{m}\notin HC(D)$ for any $m\geq M$. In particular, $D$ does not have any frequently hypercyclic algebra.
\end{proposition}

\begin{proof}
We first note that
\[
\|x\|_1 = \sum_{n=0}^\infty |x_n|, \quad x\in H(\mathbb{C})
\]
defines a continuous seminorm on $H(\mathbb{C})$. Now, let $x\in FHC(D)$ and $0<\varepsilon<1$. Then the set 
$$
A=\{n\in\mathbb N:\Vert D^{n}x\Vert_1<\varepsilon\}
$$
has positive lower density. Writing the elements of $A$ as a strictly increasing sequence $(n_{k})_{k}$, we see from Lemma \ref{lemmapld} that there exists a natural number $M$ such that 
\begin{equation}\label{equation:ineqcond2}
n_{k+1}\leq Mn_{k}
\end{equation}
for all $k\in\mathbb{N}$. 

Now assume that $x^{m}\in HC(D)$ for some $m\geq M$. Then there exists a natural number $n\geq n_1$ such that
$$
\Vert D^{n}x^{m}-e_0\Vert_1 <1-\varepsilon^m,
$$
hence
$$
1-\varepsilon^m>\Vert D^{n}x^{m}-e_{0}\Vert_1 \geq \vert n!x^{m}(n)-1\vert \geq 1- n!\vert x^{m}(n)\vert,
$$
which implies that
$$
\vert x(n)\vert>\frac{\varepsilon}{n!^{\frac{1}{m}}}.
$$

There is some $k\in\mathbb{N}$ such that $n_k\leq n <n_{k+1}$. It follows with \eqref{equation:ineqcond2} that 
\[
mn_k\geq Mn_k\geq n_{k+1}\geq n
\]
and therefore
\[
\big(n(n-1)\cdots (n-n_k+1)\big)^m \geq n!,
\]
hence 
\[
\frac{n(n-1)\cdots (n-n_k+1)}{n!^{\frac{1}{m}}}\geq 1.
\]
But then
\begin{align*}
\varepsilon 
&>\Vert D^{n_{k}}x\Vert_1 \geq \big\vert \big[D^{n_{k}}x\big](n-n_{k})\big\vert=n(n-1)\cdots (n-n_k+1) |x(n)|\\
&>\frac{n(n-1)\cdots (n-n_k+1)}{n!^{\frac{1}{m}}}\varepsilon\geq \varepsilon,
\end{align*}
which is a contradiction. This shows that no $x^m$, $m\geq M$, can be hypercyclic for $D$.
\end{proof}

In another direction one may hope that there is some weighted backward shift $B_w$ on the spaces $\ell_p$ or $c_0$ that possesses a frequently hypercyclic algebra. Recall that $B_w$ is defined as
\[
B_w(x(1),x(2),x(3),\ldots) = (w(2)x(2),w(3)x(3),w(4)x(4),\ldots),
\]
where $w=(w(n))_{n\geq 2}$ is a bounded sequence of non-zero complex numbers. In the Rolewicz case, the weights $w_n=\lambda$, $n\geq 2$, are essentially the maximal weights allowed by the continuity of the operator. Thus, by choosing smaller weights, a frequently hypercyclic algebra might turn up. Recall that the weights cannot be too small because, in $\ell_p$, $B_w$ is frequently hypercyclic if and only if
\[
\sum_{n=2}^\infty \frac{1}{|w(2)w(3)\cdots w(n)|^p} <\infty,
\]
see \cite{BaRu15}. We see, however, that for the spaces $\ell_p$, an obstruction different from the one observed for the Rolewicz operators appears when the weights are too small. 

\begin{proposition}\label{prop:nopowersBw}
Let $B_w$ be a weighted backward shift on the Banach algebra $\ell_p$, $1 \leq p < \infty$, endowed with coordinatewise multiplication. Assume that $|w(n)|\geq 1$ for all $n\geq 2$ and that, for some integer $m\geq 2$,
\[
\sum_{n=2}^\infty \frac{1}{|w(2)w(3)\cdots w(n)|^{\frac{p}{m}}} = \infty.
\]
Then, for any $x\in \ell_p$, $x^m$ cannot be frequently hypercyclic. In particular, $B_w$  does not have any frequently hypercyclic algebra.
\end{proposition}

\begin{proof}
Let $x\in \ell_p$ be such that $x^m$ is frequently hypercyclic. Let $0<\varepsilon<1$. Then there is a strictly increasing sequence $(n_{k})_{k}$ of natural numbers and some $M>0$ such that $n_k\leq Mk$ for all $k\geq 1$ and
\[
\|B_w^{n_k} x^m -e_1\| < \varepsilon,\quad k\geq 1;
\]
see the proof of Lemma \ref{lemmapld}. Then we have that
\begin{align*}
\varepsilon>\|B_w^{n_k} x^m -e_1\| &\geq \vert w(2)w(3)\cdots w(n_k+1)x^{m}(n_k+1)-1\vert\\& \geq 1- \vert w(2)w(3)\cdots w(n_k+1)x^{m}(n_k+1)\vert,
\end{align*}
hence
\[
\vert x(n_k+1)\vert>\frac{(1-\varepsilon)^\frac{1}{m}}{ |w(2)w(3)\cdots w(n_k+1)|^{\frac{1}{m}}}.
\]
Since $x\in\ell_p$ we deduce that
\[
\sum_{k=1}^\infty \frac{1}{ |w(2)w(3)\cdots w(n_k+1)|^{\frac{p}{m}}}<\infty.
\]
The fact that $|w(n)|\geq 1$ for $n\geq 2$ now implies that
\[
\sum_{k=1}^\infty \frac{1}{ |w(2)w(3)\cdots w(Mk+1)|^{\frac{p}{m}}}<\infty
\]
and thus
\[
\sum_{n=M+1}^\infty \frac{1}{ |w(2)w(3)\cdots w(n)|^{\frac{p}{m}}}\leq M \sum_{k=1}^\infty \frac{1}{ |w(2)w(3)\cdots w(Mk+1)|^{\frac{p}{m}}}<\infty,
\]
which contradicts the hypothesis.
\end{proof}

The result applies, in particular, to the family of weights $w=(w(n))_n$ given by
\[
w(n) = \Big(\frac{n}{n-1}\Big)^\alpha,\quad n\geq 2,
\]
with $\alpha> 0$. Then $B_w$ is frequently hypercyclic in $\ell_p$ if and only if $\alpha>\frac{1}{p}$, but it never admits a frequently hypercyclic algebra.

\section{The set of frequently hypercyclic vectors is always algebrable -- in some sense}\label{s-alloperators}
In this section we show that, for \textit{any} frequently hypercyclic operator $T$ on any Banach space $X$, the set $FHC(T)$ of frequently hypercyclic vectors is algebrable -- one need only define a product on $X$ suitably. In fact, the set $FHC(T)$ is always densely lineable, i.e., it contains a dense linear subspace, except zero, see \cite[Proposition 4.2]{BaGr}. The idea is then to define a product under which this subspace becomes an infinitely generated algebra. Of course, the trivial product $(x,y)\to 0$ would already do the trick. But there are somewhat less trivial products that may also be considered.

Let $X$ be a Banach space, and let $\varphi$ be a non-zero continuous linear functional on $X$ with $\|\varphi\|\leq 1$. Define
\[
y*x =\varphi(y)x,\quad x,y\in X.
\]
It is easily verified that $(X,*)$ is a (non-commutative) Banach algebra. It has the property that, for any $x\in X$ and $j\in\mathbb{N}$,
\[
x^j = \varphi(x)^{j-1}x;
\]
in particular, the algebra generated by a set coincides with the linear subspace generated by the set. Hence the algebraic structure is not of great interest, but it allows us to obtain the following.

\begin{proposition}\label{thrm:posexample}
Let $T$ be a frequently hypercyclic operator on a Banach space $X$. Let $*$ be the product on $X$ defined by a non-zero continuous linear functional $\varphi$ on $X$. Then the set $FHC(T)$ contains a dense, infinitely generated algebra, except zero. In particular, $FHC(T)$ is algebrable.
\end{proposition}

\begin{proof} By the mentioned result in \cite{BaGr}, $FHC(T)$ contains a dense linear subspace $M$, except 0. It follows from the definition of $*$ that $M$ is even an algebra under that product. Now, if it were finitely generated, then $M$ would be finite-dimensional, which is not the case.
\end{proof}

\begin{remark}
One might object that the previous product is non-commutative. The following product $(x,y)\to \varphi(x)\varphi(y)x_0$, $x,y \in X$, where $x_0$ is a fixed element of $X$ with $\|x_0\|\leq 1$, gives to $X$ the structure of a commutative Banach algebra, and if $x_0\neq 0$ is an element of a dense linear subspace of $FHC(T)$, the product defined by the functional $\varphi$ satisfies Proposition \ref{thrm:posexample}.
\end{remark}

We conclude this section by showing that in many cases we can consider a suitable product on $X$ in such a way that the set of frequently hypercyclic vectors of an operator $T$ is even strongly algebrable. Note that for the products considered above, $FHC(T)$ is very far from being strongly algebrable; indeed, for these products, $x^2$ and $x^3$ are not linearly independent.

We recall that a frequently hypercyclic subspace for an operator $T$ is a closed infinite-dimensional subspace of frequently hypercyclic vectors for $T$, except zero. For operators that possess frequently hypercyclic subspaces we refer to \cite{BoGE12}, \cite{Men15}.

\begin{theorem}\label{thrm:niceexample}
Let $T$ be a frequently hypercyclic operator on a Banach space $X$.  If $T$ admits a frequently hypercyclic subspace, then there exists a product $\times$ on $X$ so that $(X,\times)$ is a commutative Banach algebra and $FHC(T)$ is strongly algebrable.
\end{theorem}

\begin{proof}
Since $T$ admits a frequently hypercyclic subspace, there exists by Mazur's theorem a closed infinite-dimensional subspace $V$ of frequently hypercyclic vectors of $T$, except 0, that has a normalized Schauder basis. We split this Schauder basis into two subsequences $(x_{l})_{l\geq 1}$ and $(a_{r})_{r\geq 1}$. 

Let $E$ be a countable dense subset of the set of finite sequences of norm at most 1 in $\ell_\infty$. Let $\Lambda=(\lambda_{k,r})_{k,r\geq 1}$ be a matrix so that each column belongs to $E$, and each element of $E$ appears infinitely often as a column.  

For each $r\geq 1$ we define the mapping $\phi_{r}$ on $V$ as 
\begin{equation*}
\phi_{r}(x)=a_{r}^*(x)+\sum_{l=1}^\infty \lambda_{l,r} x_{l}^*(x),
\end{equation*}
where $a_r^*$ and $x_l^*$ are the coefficient functionals corresponding to the Schauder basis of $V$; note that the series is a finite sum.  Since the coefficient functionals  are linear and continuous on $V$, we have that $\phi_{r}$ is linear and continuous on $V$. By the Hahn-Banach theorem, we can extend the mappings $\phi_{r}$ to be defined on $X$. Since $\phi_r(a_r)=1=\|a_r\|$, we have that $\|\phi_r\|\geq 1$ for all $r\geq 1$.

Now, for any two points $x,y\in X$ we define 
\begin{equation*}
y\times x=\sum_{r=1}^{\infty}\frac{2^{-r}}{\Vert \phi_r\Vert^2}\phi_{r}(y)\phi_{r}(x)a_{r}.
\end{equation*}
We claim that $(X,\times)$ is a commutative Banach algebra. 

The fact that $\times$ is commutative follows directly from the definition of the multiplicative structure.

We check now that the product is associative. Given $x, y, z\in X$ we have that
\begin{equation}\label{eqtimes}
\begin{split}
z\times(y\times x)&=z\times \Big(\sum_{r=1}^{\infty}\frac{2^{-r}}{\Vert \phi_r\Vert^2}\phi_{r}(y)\phi_{r}(x)a_{r} \Big)\\
&=\sum_{s=1}^{\infty}\sum_{r=1}^{\infty}\frac{2^{-s-r}}{\Vert \phi_s\Vert^2 \Vert \phi_r\Vert^2}\phi_{s}(z)\phi_{r}(y)\phi_{r}(x)\phi_{s}(a_{r})a_{s}\\
&=\sum_{r=1}^{\infty}\frac{2^{-2r}}{\Vert \phi_r\Vert^4}\phi_{r}(z)\phi_{r}(y)\phi_{r}(x)a_{r}\quad\text{(since $\phi_{s}(a_{r})=\delta_{r,s}$)},
\end{split}
\end{equation}
and similarly
\begin{align*}
(z\times y)\times x &= \Big(\sum_{r=1}^{\infty}\frac{2^{-r}}{\Vert \phi_r\Vert^2}\phi_{r}(z)\phi_{r}(y)a_{r}\Big) \times x\\
&= \sum_{s=1}^{\infty}\sum_{r=1}^{\infty}\frac{2^{-s-r}}{\Vert \phi_s\Vert^2\Vert \phi_r\Vert^2}\phi_{r}(z)\phi_{r}(y)\phi_{s}(a_{r})\phi_{s}(x)a_{s}\\
&= \sum_{r=1}^{\infty}\frac{2^{-2r}}{\Vert \phi_r\Vert^4}\phi_{r}(z)\phi_{r}(y)\phi_{r}(x)a_{r}. 
\end{align*}
Therefore $z\times(y\times x)=(z\times y)\times x$ and $\times$ is associative on $X$. Also, for all $x,y\in X$,
\begin{align*}
\Vert y\times x\Vert& \leq \sum_{r=1}^{\infty}\frac{2^{-r}}{\Vert \phi_r\Vert^2}\vert\phi_{r}(y)\vert \vert\phi_{r}(x)\vert\Vert a_{r}\Vert\\
&\leq\Vert y\Vert \Vert x \Vert \sum_{r=1}^{\infty}2^{-r} =\Vert y\Vert \Vert x\Vert.
\end{align*}
Hence, $(X,\times)$ is a commutative  Banach algebra.

We note for later use that, by \eqref{eqtimes} generalized to higher products and applied to points in the sequence $(x_{l})_{l\geq 1}$, we have for any $j_1,j_2,\ldots, j_s\in\mathbb N$, $s\geq 2$,
\begin{align*}
x_{j_1}\times x_{j_2}\times \cdots\times x_{j_s}&=\sum_{r=1}^\infty \frac{2^{-r(s-1)}}{\Vert \phi_r\Vert^{2s-2}}\phi_r(x_{j_1})\cdots \phi_r(x_{j_s})a_r\\
&=\sum_{r=1}^\infty \frac{2^{-r(s-1)}}{\Vert \phi_{r}\Vert^{2s-2}}\lambda_{j_1,r}\cdots \lambda_{j_s,r}a_{r},
\end{align*}
and hence for any multi-index $\beta=(\beta_1,\beta_2,\ldots, \beta_s)\in\mathbb N_0^s$, $s\geq 1$, with $|\beta|\geq 2$,
\begin{equation}\label{basicproducts}
x_{1}^{\beta_1}\times x_{2}^{\beta_2}\times \cdots\times x_{s}^{\beta_s}=\sum_{r=1}^\infty \frac{2^{-r(|\beta|-1)}}{\Vert \phi_r\Vert^{2|\beta|-2}}\lambda_{1,r}^{\beta_1}\cdots \lambda_{s,r}^{\beta_s}a_{r}.
\end{equation}

Let $\mathcal B$ be the algebra generated by the sequence $(x_{l})_{l\geq 1}$. It follows from the definition of the product that any element in $\mathcal B$ is the sum of a linear combination of certain $x_l$ and a convergent series in multiples of the $a_l$, so that 
$\mathcal B\subset V$. This implies that any non-zero element of $\mathcal B$ lies in $FHC(T)$. 

To finish we only need to show that the sequence $(x_l)_{l\geq 1}$ is algebraically independent. Assume that this is not the case. Then there exist a finite set of multi-indices $I\subset \mathbb N_0^s$ and complex numbers $c_\beta$, $\beta\in I$, not all of them zero, such that
\begin{equation}\label{algebraiccombination}
0=\sum_{\beta\in I,\beta\neq 0}c_{\beta}x_{1}^{\beta_{1}}\times\cdots\times x_{s}^{\beta_{s}} =\sum_{l=1}^{s}\eta_{l}x_{l}+\sum_{r=1}^{\infty}\gamma_{r}a_{r}
\end{equation}
for some complex numbers $\eta_{l}$ and $\gamma_r$, $l,r\geq 1$, where the final series converges in $X$.

Note that, since the union of $(x_{l})_{l}$ and $(a_{r})_{r}$ is a Schauder basis for $V$, we have that $\eta_{l}=\gamma_{r}=0$ for all $l, r\geq 1$. We will show that this leads to a contradiction. In fact, it is clear from the construction of the product that each coefficient $c_{\beta}$ with $\vert \beta\vert=1$ coincides with one of the $\eta_{l}$, $l\geq 1$, hence all of these $c_{\beta}$ must be zero. We show now that not all $\gamma_{r}$ can be equal to zero.

In order to see this, consider the complex polynomial $P$ defined in $\mathbb C^s$ by
\begin{equation*}
P(z_1,\ldots,z_s)=\sum_{\beta\in I,\beta\neq 0}c_{\beta}z_1^{\beta_{1}}\cdots z_s^{\beta_{s}}.
\end{equation*}
Since some $c_\beta\neq 0$, $P$ is non-zero. Assume that $P$ has degree $M$. We denote by $P_{j}, \ldots, P_{M}$ the homogeneous parts of $P$, so $P(z_1,\ldots,z_s)=\sum_{m=j}^M P_{m}(z_1,\ldots,z_s)$, and $P_{j}$ is a non-zero $j$-homogeneous polynomial. Note that $j>1$ by our previous argument.

It follows from \eqref{basicproducts} that, for $j\leq m\leq M$,
\begin{align*}
\sum_{\beta\in I,|\beta|=m}c_{\beta}x_{1}^{\beta_{1}}\times\cdots\times x_{s}^{\beta_{s}} &= 
\sum_{r=1}^\infty \frac{2^{-r(m-1)}}{\Vert \phi_{r}\Vert^{2m-2}}\Big(\sum_{\beta\in I,|\beta|=m}c_{\beta}
\lambda_{1,r}^{\beta_1}\cdots \lambda_{s,r}^{\beta_s}\Big)a_{r}\\
&= \sum_{r=1}^\infty \frac{2^{-r(m-1)}}{\Vert \phi_{r}\Vert^{2m-2}} P_m(\lambda_{1,r},\ldots, \lambda_{s,r})a_{r},
\end{align*}
and hence
\[
\sum_{\beta\in I,\beta\neq 0}c_{\beta}x_{1}^{\beta_{1}}\times\cdots\times x_{s}^{\beta_{s}} = 
\sum_{r=1}^\infty \Big(\sum_{m=j}^M\frac{2^{-r(m-1)}}{\Vert \phi_{r}\Vert^{2m-2}}P_m(\lambda_{1,r},\ldots, \lambda_{s,r})\Big)a_{r}.
\]
Thus \eqref{algebraiccombination} implies that
\begin{equation}\label{eqgamma}
\gamma_r=\sum_{m=j}^M\frac{2^{-r(m-1)}}{\Vert \phi_{r}\Vert^{2m-2}}P_m(\lambda_{1,r},\ldots, \lambda_{s,r}), \ r\geq 1.
\end{equation} 

Now, since the polynomial $P_{j}$ is not identically zero, there exists, by the definition of $\Lambda$, some $r\geq 1$ with
$$
\vert P_{j}(\lambda_{1,r},\ldots,\lambda_{s,r})\vert>0.
$$
Consider the columns $(\lambda_{k,r_{n}})_k$, $n\geq 1$, of $\Lambda$ that coincide with $(\lambda_{k,r})_{k}$.
Then we obtain that
\begin{align*}
\Big\vert \sum_{m=j}^M & \frac{2^{-r_n(m-1)}}{\Vert \phi_{r_n}\Vert^{2m-2}}P_{m}(\lambda_{1,r_{n}},\ldots,\lambda_{s,r_{n}})\Big\vert = \Big\vert \sum_{m=j}^M  \frac{2^{-{r_n}(m-1)}}{\Vert \phi_{r_n}\Vert^{2m-2}}P_{m}(\lambda_{1,r},\ldots,\lambda_{s,r})\Big\vert\\
&\geq \frac{2^{-{r_n}(j-1)}}{\Vert \phi_{r_n}\Vert^{2j-2}}\Big(\vert P_{j}(\lambda_{1,r},\ldots,\lambda_{s,r})\vert -\!\sum_{m=j+1}^M \frac{2^{-{r_n}(m-j)}}{\Vert \phi_{r_n}\Vert^{2(m-j)}}\vert  P_{m}(\lambda_{1,r},\ldots,\lambda_{s,r})\vert\Big)>0
\end{align*}
for all numbers $n\geq 1$ that are big enough, where we have used that $\|\phi_r\|\geq 1$ for all $r\geq 1$. 

Therefore, by \eqref{eqgamma}, we have that $\gamma_{r_n}$ is non-zero for some $n\geq 1$, which is the desired contradiction. 
\end{proof}


\end{document}